\documentclass[10pt]{amsproc}

\usepackage{amssymb}
\usepackage[utf8]{inputenc}
\usepackage[T1]{fontenc}
\usepackage{lmodern}
\usepackage{mathtools}
\usepackage{mathrsfs}
\usepackage{graphicx}

\usepackage{tikz}
\usepackage{tikz-cd}

\usetikzlibrary{arrows.meta,calc,decorations.markings,intersections}
\usepackage{xcolor}
\usepackage{hyperref}
\hypersetup{colorlinks=true,linkcolor=magenta,citecolor=blue,urlcolor=blue}

\usepackage{enumitem}

\newcommand{\textupbf}[1]{\textbf{\textup{#1}}}

\allowdisplaybreaks

\SetEnumitemKey{auto}{label=\textupbf{(\arabic*)},ref=(\arabic*),leftmargin=*}

\newcommand{\rmop}[1]{\operatorname{#1}}

\numberwithin{equation}{section}
\numberwithin{figure}{section}

\newtheorem{lemma}{Lemma}[section]
\newtheorem{proposition}{Proposition}[section]
\newtheorem{theorem}{Theorem}[section]
\newtheorem{corollary}{Corollary}[section]

\theoremstyle{definition}

\newtheorem{remark}{Remark}[section]

\newcommand{\frakP}{\mathfrak{P}}

\newcommand{\bbF}{\mathbb{F}}

\newcommand{\bbQ}{\mathbb{Q}}

\newcommand{\bbZ}{\mathbb{Z}}

\newcommand{\calH}{\mathcal{H}}

\newcommand{\calO}{\mathcal{O}}

\newcommand{\scrL}{\mathscr{L}}

\newcommand{\rmH}{\mathrm{H}}

\newcommand{\Z}{\mathbb{Z}}
\newcommand{\Q}{\mathbb{Q}}

\newcommand{\End}{\operatorname{End}}
\newcommand{\Hom}{\operatorname{Hom}}
\newcommand{\NS}{\operatorname{NS}}
\newcommand{\tr}{\operatorname{tr}}
\newcommand{\trd}{\operatorname{trd}}
\newcommand{\nrd}{\operatorname{nrd}}
\newcommand{\rank}{\operatorname{rank}}
\newcommand{\adj}{\operatorname{adj}}
\newcommand{\GL}{\operatorname{GL}}
\newcommand{\id}{\operatorname{id}}

\usepackage[backend=biber,style=alphabetic]{biblatex}
\title[Small degree isogenies between conjugate surfaces]{Small degree isogenies between conjugate principally polarized superspecial abelian surfaces}

\author{Eda Kırımlı}
\email{e.kirimli@bham.ac.uk}
\address{School of Computer Science, University of Birmingham}
\thanks{E. Kırımlı (UoB) is supported by EPSRC through grant number EP/V011324/1}

\author{Lam L. Pham}
\email{lam.pham@ulb.be}
\address{Department of Computer Science, Universit\'e Libre de Bruxelles}
\thanks{L.L. Pham (ULB) is supported by the FRS-FNRS through grant number T.0238.25}

\DeclareMathOperator{\Herm}{Herm}
\DeclareMathOperator{\Nrd}{Nrd}
\DeclareMathOperator{\Trd}{Trd}

\newcommand{\medsim}{\ensuremath\mathrel{\scalebox{1}{\ensuremath\sim}}}
\newcommand{\stosim}[1][]{\stackrel{\medsim}{\smash{\xrightarrow[#1]{\hspace{0.5em}}}\rule{0pt}{0.2ex}}}

\begin{document}

\begin{abstract}
Let \(p>3\)  be a prime number.
We prove that every principally polarized superspecial abelian surface over \(\bbF_{p^{2}}\) with \(p^{2}\)-Frobenius \([-p]\) admits a separable polarized isogeny to its Frobenius conjugate with multiplier at most \((p^{3}/2)^{1/5}\).
This generalizes a recent result of Aubry, Oyono, and Vincent, who proved that every supersingular elliptic curve defined over \(\bar{\bbF}_{p}\) admits an isogeny to its Frobenius conjugate with degree at most \((p/2)^{1/3}\).
\end{abstract}

\maketitle

\tableofcontents

\section{Introduction and main result}

Let \(p>3\) be a prime, and let \( k / k _{0}=\bbF_{p^{2}}/\bbF_{p}\).
Let \(\sigma: k \to  k \), \(\sigma(x)=x^{p}\) be the non-trivial element of \(\rmop{Gal}( k / k _{0})\).
For a \( k \)-scheme \(X\), let \(X^{\sigma}\) be its Galois conjugate.
We have the relative Frobenius morphisms \(F_{X}: X \to X^{\sigma}\) and \(F_{X^{\sigma}}:X^{\sigma} \to X\).
Write \(\pi_{X,p^{2}}=F_{X^{\sigma}}\circ F_{X}\) for the \(p^{2}\)-Frobenius endomorphism.
We consider surfaces over \( k \) such that \(\pi_{X,p^{2}}=[-p]_{X}\), where \([-p]_{X}\) denotes multiplication by \((-p)\) on the surface \(X\).
A \emph{polarization} is a homomorphism \(\theta_{X}:X\to X^{\vee}\), where \(X^{\vee}\) denotes the dual abelian surface of \(X\), induced by an ample line bundle, and it is a principal polarization if it is, in addition, an isomorphism.
An isomorphism of polarized surfaces (or polarized isomorphisms, for short) \(u:(Y,\theta_{Y})\to (Z,\theta_{Z})\) is an isomorphism satisfying \(u^{\vee}\theta_{Z} u = \theta_{Y}\).

For a homomorphism \(u:Y \to Z\) between principally polarized abelian surfaces, the \emph{polarized adjoint} is defined by the commutative square:
\[
\begin{tikzcd}
Z \dar[swap]{\theta_{Z}} \rar{u^{\dagger}} & Y \dar{\theta_{Y}}\\
Z^{\vee} \rar[swap]{u^{\vee}} & Y^{\vee}
\end{tikzcd}
\qquad
u^{\dagger} = \theta_{Y}^{-1} \circ u^{\vee} \circ \theta_{Z}
\]
(see, e.g., \cite{DiemNaumann2003}).
For \(u: Y \to Z\) and \(v: Z \to W\), this implies in particular that \((v \circ u)^{\dagger} = u^{\dagger} \circ v^{\dagger}\) and \((u^{\dagger})^{\dagger}=u\).
Acting on \emph{endomorphisms}, \({}^{\dagger}\) is the \emph{Rosati involution}.
Finally, we say that a polarized isogeny \(u:Y \to Z\) has \emph{multiplier} or \emph{reduced degree} \(N\in \bbZ_{>0}\) if
\begin{equation}
\label{eq:multiplier}
u^{\vee} \theta_{Z} u = N \theta_{Y},\quad\text{or, equivalently},\quad
u^{\dagger}u=[N]_{Y}.
\end{equation}
If \(Y\) is a surface, then \(\deg(u)=N^{2}\) since \(\deg(u^{\dagger})=\deg(u)\).

We can now state the main theorem.

\begin{theorem}
\label{theorem:A}
Let \((A,\theta)_{/ k }\) be a principally polarized superspecial abelian surface whose \(p^{2}\)-Frobenius endomorphism is \([-p]_{A}\), that is, \(F_{A^{\sigma}}\circ F_{A} = [-p]_{A}\).
Then, there exists an integer \(N>0\) and a separable polarized \( k \)-isogeny \(\phi:(A,\theta) \to (A^{\sigma},\theta^{\sigma})\) such that
\[
\phi^{\dagger}\phi=[N]_{A},\quad
1\leq N\leq \left(\frac{p^{3}}{2}\right)^{1/5}<p^{3/5}<p,\quad \deg\phi=N^{2}.
\]
\end{theorem}

This result generalizes to principally polarized superspecial surfaces, a recent result proved by Aubry, Oyono, and Vincent \cite{AubryOyonoVincent2026} for supersingular elliptic curves.

\subsection*{Acknowledgments}

After the first version of this article by L.P. appeared on arXiv, L.P. learned that E.K. had also previously independently obtained the same result through Kani's refined Humbert invariant.
Although the proofs are morally equivalent, they differ substantially in content and tools, and in the interest of the wider mathematical community, we have opted to rewrite the article to incorporate both viewpoints.

\subsection*{Organization of the article}

In Part \ref{part:1}, we shall prove Theorem \ref{theorem:A} directly using quaternionic coordinates.
In Part \ref{part:2}, we introduce the refined Humbert invariant and give another proof of Theorem \ref{theorem:A}.

\part{Quaternionic coordinates}
\label{part:1}

\section{Quaternionic coordinates for abelian surfaces}

For basic facts about elliptic curves, we refer the reader to the books of Voight \cite[Chapter 42]{Voight2021} and Silverman \cite{Silverman2009}.
For abelian surfaces, we refer to Mumford's book \cite{Mumford2008} and to Milne's article \cite{Milne1986}.

\subsection{Elliptic curves and endomorphism rings}

We use the following classical facts and fix the following model.
Let \(E_{0}{}_{/\bbF_{p}}\) be a fixed supersingular elliptic curve.
There exists a supersingular elliptic curve \(E_{0}{}_{/\bbF_{p}}\) whose \(p\)-Frobenius \(\pi_{0}\) satisfies \(\pi_0^{2}=[-p]_{E_{0}}\).
Every supersingular elliptic curve over \(\bar{ k }\) has a \( k \)-model with \(p^{2}\)-Frobenius \([-p]\), and the choice of this \( k \)-model follows from \(E_{0}\).
This follows from Deuring \cite{Deuring1941}.

We denote by \(E\), resp. \(\pi\) the base change of \(E_{0}\), resp. \(\pi_{0}\), to \( k \).
Then, we have canonically \(E^{\sigma}=E\) with Frobenius \(F_{E}=\pi\).
Let
\[
\calO = \rmop{End}_{ k }(E) = \rmop{End}_{\bar{ k }}(E_{\bar{ k }}),\quad
B=\calO \otimes_{\bbZ}\bbQ.
\]
By Deuring's theorem, \(B\) is the quaternion algebra over \(\bbQ\) ramified exactly at \(\{p,\infty\}\), often denoted \(B_{p,\infty}\) in the literature, and \(\calO\) is a maximal order of \(B\).

For an isogeny of elliptic curves \(f:E \to E^{\prime}\) of degree \(n\), there exists a unique dual isogeny \(\hat{f}:E^{\prime} \to E\) such that \(\hat{f}\circ f = [n]_{E}\) and \(f\circ \hat{f}=[n]_{E^{\prime}}\).
This dual isogeny also has degree \(n\).
It is related to the dual abelian surface as follows.
An elliptic curve \(E\) has a canonical principal polarization \(\theta_{E}:E \to E^{\vee}\).
On the other hand, the dual homomorphism \(f^{\vee}:(E^{\prime})^{\vee} \to E^{\vee}\) satisfies \(\hat{f}=\theta_{E}^{-1} \circ f^{\vee} \circ \theta_{E^{\prime}}\), so that \(\hat{f}\) is the polarized adjoint for the canonical polarization.

On \(B\), quaternion conjugation \(a\mapsto \bar{a}\) corresponds on \(\calO\) to the elliptic curve dual, and the reduced norm restricts to degree.
The reduced norm and reduced trace on \(B\) are \(\rmop{trd}(a)=a+\bar{a}\) and \(\rmop{nrd}(a)=a\bar{a}\) respectively.
Then, we have
\[
\pi^{2}=-p,\quad
\bar{\pi}=-\pi,\quad
\rmop{nrd}(\pi)=p.
\]
For \(a\in \calO\), we have \(\rmop{nrd}(a)=\deg(a)\).

\subsection{Coordinates for a principally polarized surface}
\label{sec:Coordinates}
We will change coordinates for principal polarizations but this applies to arbitrary homomorphisms as follows.
Let \(\phi:A \to A^{\vee}\) be a morphism.
We fix an isomorphism \(\iota:A \stosim E^{2}\) of unpolarized surfaces;
it induces an isomorphism of abelian groups
\[
\rmop{Hom}_{ k }(A,A^{\vee}) \to \rmop{Hom}_{ k }(E^{2},(E^{2})^{\vee}),\quad
\phi\mapsto (\iota^{-1})^{\vee} \circ \phi \circ \iota^{-1}.
\]

\medskip

Let \((A,\theta)\)  be a principally polarized superspecial abelian surface over \( k \).
By a theorem of Deligne theorem (e.g., see \cite[Theorem 6.2]{Ogus1979} and \cite[Theorem 3.5]{Shioda1979}), \(A_{\bar{ k }}\) is isomorphic to \(E_{\bar{ k }}^{2}\).
Since both \(A\) and \(E^{2}\) have \(p^{2}\)-Frobenius \([-p]\), any such geometric isomorphism commutes with \(p^{2}\)-Frobenius, thus is defined over \( k =\bbF_{p^{2}}\) (see \cite[p.~43]{IbukiyamaKatsura1994}).
We thus fix an unpolarized \( k \)-isomorphism \(\iota:A \to E^{2}\).
Let \(\theta_{0}:E^{2} \to (E^{2})^{\vee}\) be the \emph{product principal polarization} on \(E^{2}\), that is, \(\theta_{0}=\theta_{E}\times \theta_{E}\), and define \(\theta_{E^{2}}\) as the unique homomorphism satisfying \(\iota^{\vee} \circ \theta_{E^{2}} \circ \iota = \theta\), i.e., we have a commutative diagram:
\[
\begin{tikzcd}
A \dar[swap]{\iota} \rar{\theta} & A^{\vee} \dar{(\iota^{-1})^{\vee}} \\
E^{2} \rar[swap]{\theta_{E^{2}}} & (E^{2})^{\vee}
\end{tikzcd}
\qquad
\theta_{E^{2}}=(\iota^{-1})^{\vee} \circ \theta \circ \iota^{-1}.
\]
We then define the endomorphism \(h:E^{2}\to E^{2}\) by \(\theta_{E^{2}}=\theta_{0}\circ h\), that is,
\begin{equation}
\label{eq:IKO-matrix}
h=\theta_{0}^{-1} \circ \theta_{E^{2}}:E^{2} \to E^{2},\quad h\in \rmop{End}_{ k }(E^{2})=\rmop{M}_{2}(\calO).
\end{equation}
The endomorphism \(h\) compares these two polarizations on the same underlying unpolarized surface \(E^{2}\).

\subsection{The Rosati involution}

Given \((A,\theta)\), the Rosati involution is obtained from the polarized adjoint.
Accordingly, the \emph{Rosati involution} associated with \(\theta\) is
\[
\rmop{End}_{ k }(A) \to \rmop{End}_{ k }(A),\quad
u\mapsto u^{\dagger_{\theta}} = \theta^{-1} u^{\vee} \theta.
\]
Now, let \(z=\iota u \iota^{-1}\in \rmop{M}_{2}(B)\).
We can compute its adjoint with respect to the transported polarization \(\theta_{E^{2}}\) and obtain
\[
\theta_{E^{2}}^{-1} \circ z^{\vee} \circ \theta_{E^{2}} = \iota \circ (u^{\dagger_{\theta}})\circ \iota^{-1}.
\]

Recall that for elliptic curves and an endomorphism \(f:E\to E\), the elliptic curve dual isogeny \(\hat{f}\) coincides with the polarized adjoint for the canonical polarization, that is, if \(f\in \rmop{End}_{ k }(E)\), then \(\hat{f}=\theta_{E}^{-1} \circ f^{\vee} \circ \theta_{E}\), which is precisely the Rosati involution.
Thus, on the quaternion algebra \(B\), if \(a\in B\), then \(\theta_{E}^{-1} \circ a^{\vee} \circ \theta_{E} = \bar{a}\).
Under the identification \((E^{2})^{\vee}\cong (E^{\vee})^{2}\), if \(z\in \rmop{End}_{ k }(E^{2})\), writing \(z=(z_{ij})\), we have \((z^{\vee})_{ij}=(z_{ji}^{\vee})\).
Therefore, \(\theta_{0}^{-1} z^{\vee} \theta_{0} = z^{*}\).
Since \(\theta_{E^{2}}=\theta_{0} \circ h\), we obtain
\[
\iota \circ (u^{\dagger_{\theta}})\circ \iota^{-1}
 = h^{-1}z^{*}h.
\]
In our chosen matrix coordinates, the Rosati involution becomes the map
\[
\dagger_{h}: \rmop{M}_{2}(B) \to \rmop{M}_{2}(B),\quad
z\mapsto z^{\dagger_{h}} = h^{-1}z^{*}h.
\]
Back to \(u:A\to A\), since \(z=\iota u \iota^{-1}\), the relation is \(\iota\circ (u^{\dagger_{\theta}})\circ \iota^{-1} = (\iota \,u\, \iota^{-1})^{\dagger_{h}}\).
If \(\Psi:\rmop{End}_{k}(A) \to \rmop{M}_{2}(B)\) is defined by \(\Psi(u)=\iota u \iota^{-1}\), we thus have the following diagram:
\begin{equation}
\label{eq:Rosati-matrix}
\begin{tikzcd}
\rmop{End}_{k}(A) \dar[swap]{\Psi} \rar{\dagger_{\theta}} &\rmop{End}_{k}(A) \dar{\Psi}\\
\rmop{M}_{2}(\calO) \rar[swap]{\dagger_{h}} & \rmop{M}_{2}(\calO)
\end{tikzcd}
\quad
\Psi(u^{\dagger_{\theta}})=\Psi(u)^{\dagger_{h}}.
\end{equation}
We observe that it is indeed an involution on \(\rmop{M}_{2}(B)\), since \(x^{\dagger_{h}\dagger_{h}}=x\) (this follows from the fact that \(h\) is hermitian).

Then, we have the so-called ``Ibukiyama-Katsura-Oort (IKO) correspondence'' (\cite{IbukiyamaKatsuraOort1986}):

\begin{lemma}[IKO correspondence]
\label{lemma:IKO}
The matrix \(h\in \rmop{M}_{2}(\calO)\) defined by \eqref{eq:IKO-matrix} is a positive definite, unimodular quaternionic Hermitian matrix:
\[
h=\begin{pmatrix}
s & r \\ \bar{r} & t
\end{pmatrix},\quad
s, t\in \bbZ_{>0},\quad
r\in \calO,\quad
st-\rmop{nrd}(r)=1.
\]
\end{lemma}

Conjugation by \(\iota\) induces the ring isomorphism
\[
\rmop{End}_{ k }(A) \to \rmop{M}_{2}(\calO),\quad
u\mapsto \iota\, u \, \iota^{-1}.
\]

\subsection{Galois action and the ramified ideal} \label{sec:GaloisRamified}

For a \( k \)-morphism \(u:Y \to Z\), naturality of the \emph{relative} Frobenius means that the following diagram commutes:
\[
\begin{tikzcd}
Y \rar{u} \dar[swap]{F_{Y}} & Z \dar{F_{Z}} \\
Y^{\sigma} \rar[swap]{u^{\sigma}} & Z^{\sigma}
\end{tikzcd}
\qquad
F_{Z} u = u^{\sigma} F_{Y}.
\]
For \(a\in \calO\), our chosen \(\bbF_{p}\)-model (from \(E_{0}\)) gives a canonical identification \(E^{\sigma}=E\) with \(F_{E}=\pi\).
The Galois action now gives
\[
\pi a = a^{\sigma} \pi,\quad a^{\sigma} = \pi a \pi^{-1},\quad
\pi \calO \pi^{-1} = \calO.
\]
Therefore, we can define the two-sided ideal
\[
\frakP = \pi \calO = \calO \pi.
\]
Since \(\pi^{2}=-p\), \(\bar{\pi}=-\pi\), and \(\bar{\calO}=\calO\), we obtain
\[
\frakP^{2} = p\calO,\quad \bar{\frakP}=\frakP.
\]
Since \(p\calO=\frakP^{2}\subset \frakP\) and \(1\notin\frakP\) (otherwise \(1=\pi a\) implies \(1=p\nrd(a)\) with \(\nrd(a)\in \bbZ\)), we have \(p\bbZ\subset \frakP\cap \bbZ\subsetneq \bbZ\).
Since the ideal \(p\bbZ\) is maximal in \(\bbZ\), we get \(\frakP\cap \bbZ=p\bbZ\).
And since \(\frakP\subset \calO\) and \(\calO\cap \bbQ=\bbZ\), we get
\[
\frakP\cap \bbQ = p\bbZ.
\]
It follows that for any \(u\in \frakP\),
\[
\rmop{trd}(u)\in \frakP\cap\bbQ=p\bbZ,\quad \rmop{nrd}(u)\in \frakP^{2}\cap \bbQ = p\bbZ.
\]
Finally, multiplication by \(\pi\) induces an isomorphism of additive groups
\(
\calO/\frakP \stosim \frakP / p\calO,
\)
so \([\calO:\frakP]=[\frakP:p\calO]\).
From the inclusions \(p\calO\subset \frakP\subset \calO\), we get
\[
p^{4} = [\calO:p\calO]
= [\calO:\frakP]^{2}.
\]
Thus, \([\calO:\frakP]=p^{2}\).

\subsection{Frobenius in quaternionic coordinates}\label{sec:Frob}

Let \(\Pi=\pi I_{2}\in \rmop{M}_{2}(\calO)=\rmop{End}_{ k }(E^{2})\) and \(F=F_{A^{\sigma}}:A^{\sigma} \to A\).
We use the coordinates \(\iota:A \to E^{2}\) and \(\iota^{\sigma}:A^{\sigma} \to E^{2}\).
Naturality gives the commutative diagrams
\[
\begin{tikzcd}
A \dar[swap]{\iota}\rar{F_{A}} & A^{\sigma} \dar{\iota^{\sigma}} \rar{F} & A \dar{\iota}\\
E^{2} \rar[swap]{\Pi} & E^{2} \rar[swap]{\Pi} & E^{2}
\end{tikzcd}\qquad
\iota^{\sigma} F_{A} = \Pi \iota,\quad \iota F_{A^{\sigma}} = \Pi \iota^{\sigma}.
\]
If \(\phi:A \to A^{\sigma}\), then we consider \(\iota^{\sigma}\circ \phi \circ \iota^{-1}\in \rmop{End}(E^{2})\).
Define \(\gamma=\iota^{\sigma} \phi \iota^{-1}\in \rmop{End}(E^{2})\).
Thus, the matrix of \(F\phi\in \rmop{End}_{ k }(A)\) is
\(
\iota(F\phi)\iota^{-1} = \Pi \gamma.
\)
This means that we have a diagram
\begin{equation}
\label{eq:big-diagram}
\begin{tikzcd}[row sep=3em,column sep=6em]
\rmop{Hom}_{ k }(A,A^{\sigma}) \dar[swap]{\phi \mapsto \iota^{\sigma} \phi \iota^{-1}}\rar{\phi \mapsto F\circ \phi} & \rmop{End}_{ k }(A) \dar{v \mapsto \iota v \iota^{-1}} \\
\rmop{M}_{2}(\calO) \rar[swap]{\gamma \mapsto \Pi \gamma} & \rmop{M}_{2}(\calO).
\end{tikzcd}
\end{equation}
The vertical maps are isomorphisms, and the lower map is injective with image \(\Pi\, \rmop{M}_{2}(\calO)=\rmop{M}_{2}(\pi \calO) = \rmop{M}_{2}(\frakP)\).
Thus, every \(x\in \rmop{M}_{2}(\frakP)\) determines an integral matrix \(\gamma = \Pi^{-1}x\in \rmop{M}_{2}(\calO)\), and hence a homomorphism \((\iota^{\sigma})^{-1} \gamma \iota:A \to A^{\sigma}\).

Thus, \(\theta^{\sigma}\) is transported to \(\theta_{0}h^{\sigma}\).
Naturality of relative Frobenius, applied to \(h:E^{2}\to E^{2}\) gives
\(h^{\sigma} = \Pi h \Pi^{-1}\).

\subsection{Multipliers}

Recall the definition \eqref{eq:multiplier} for the multiplier or reduced degree of a polarized isogeny \(\phi:(Y,\theta_{Y}) \to (Z,\theta_{Z})\).
Here, we consider the homomorphism \(\phi:(A,\theta) \to (A^{\sigma},\theta^{\sigma})\) for which \eqref{eq:multiplier} becomes
\[
\phi^{\vee} \theta^{\sigma} \phi = N\, \theta.
\]
Let \(\gamma=\iota^{\sigma}\phi \iota^{-1}\).
Thus, the multiplier condition, in these coordinates, is
\[
\begin{tikzcd}
E^{2} \dar[swap]{N\,\theta_{0}h}\rar{\gamma} & E^{2} \dar{\theta_{0}h^{\sigma}}\\
(E^{2})^{\vee} & (E^{2})^{\vee} \lar{\gamma^{\vee}}
\end{tikzcd}
\qquad
\gamma^{\vee} \theta_{0} h^{\sigma} \gamma = N \theta_{0}h.
\]
Now, we use the product Rosati involution.
Since \(\gamma\in \rmop{M}_{2}(B)\), we know that \(\gamma^{*}=\theta_{0}^{-1} \gamma^{\vee} \theta_{0}\), that is, \(\gamma^{\vee}\theta_{0}=\theta_{0}\gamma^{*}\).
Therefore, the multiplier equation becomes
\begin{equation}
\label{eq:conjugate-multiplier-matrix}
\gamma^{*}h^{\sigma}\gamma=Nh.
\end{equation}

Since \(x\) is the matrix of \(F\phi:A\to A\), we can rewrite the multiplier equation.
Recall that \(\Pi=\pi I_{2}\), \(\bar{\pi}=-\pi\), \(\pi^{2}=-p\).
Thus, \(\Pi^{*}=\bar{\pi}I_{2}=-\Pi\), and \(\Pi^{2}=-pI_{2}\), so \(\Pi^{-1}=-\Pi/p\).
Therefore,
\[
\Pi^{*}h\Pi
=p h^{\sigma},
\]
since \(h^{\sigma}=\Pi h \Pi^{-1}\).
Now, the multiplier condition \(\gamma^{*}h^{\sigma} \gamma = Nh\) holds if and only if
\begin{equation}
\label{eq:conjugate-multiplier-refined-matrix}
x^{*}hx=pNh,\quad\text{or, equivalently},\quad
x^{\dagger_{h}}x=pNI_{2},
\end{equation}
since \(x^{\dagger_{h}}=h^{-1}x^{*}h\).
This will allow us to establish the multiplier condition directly in \(\iota\)-coordinates.

\section{The rank five lattice}

\subsection{Subset of Rosati fixed elements}

Let
\[
\rmop{Herm}_{2}(B) = \left\{
\begin{pmatrix}
a & b \\ \bar{b} & c
\end{pmatrix} : a,c\in \bbQ,\, b\in B
\right\}= \{H\in \rmop{M}_{2}(B) : H^{*}=H\}.
\]
We define the \emph{Hermitian determinant}
\[
{\det}_{\rmH}:\Herm_{2}(B) \to \bbQ,\quad
\begin{pmatrix}
a & b \\ \bar{b} & c
\end{pmatrix}\mapsto ac-\rmop{nrd}(b),
\]
and let \(f=-{\det}_{\rmH}\).

Let \(h\in \rmop{M}_{2}(\calO)\), with \({\det}_{\rmH}(h)=1\), a positive definite Hermitian matrix. We denote by $\Trd$ and $\Nrd$ the reduced trace and norm on $\rmop{M}_{2}(B)$.
We now consider the following subspace of \(\rmop{M}_{2}(B)\),
\[
S_{h}
 = \{ x\in \rmop{M}_{2}(B) \mid x^{\dagger_{h}} = x,\, \rmop{Trd}(x) = 0\}.
\]
This is a \(\bbQ\)-vector subspace of \(\rmop{M}_{2}(B)\).
Note that \(x^{\dagger_{h}}=x\) if and only if \(x^{*}h=hx\).

\begin{proposition}
Let \(S_{h}\) be as above.
Then, \(\dim_{\bbQ}(S_{h})=5\).
The quadratic form
\[
q_{0}:S_{h} \to \bbQ,\quad x\mapsto \tfrac{1}{4} \rmop{Trd}(x^{2})
\]
is positive definite.
For every \(x\in S_{h}\), we have
\begin{enumerate}[auto]
\item \(x^{2}=q_{0}(x)I_{2}\),
\item \(hx\in \Herm_{2}(B)\),
\item \(q_{0}(x)=f(hx)\).
\end{enumerate}
\end{proposition}

\begin{proof}
Property (2) is immediate.
Since \(x^{*}h=hx\) for \(x\in S_{h}\), we have \((hx)^{*}=x^{*}h^{*}=x^{*}h=hx\).

Let \(C=\begin{psmallmatrix}
s & r\\ 0 & 1
\end{psmallmatrix}\in \rmop{GL}_{2}(B)\).
Since \(st-\nrd(r)=1\), we compute
\[
C^{*}C=\begin{pmatrix}
s^{2} & sr\\ s\bar{r} & \nrd(r)+1
\end{pmatrix}=sh.
\]
Let \(Y=CxC^{-1}\).
Then, \(s^{-1}C^{*}=hC^{-1}\), so
\begin{align*}
x^{*}h = hx &\quad \Leftrightarrow\quad
(C^{-1}YC)^{*}h = hC^{-1}YC \\
&\quad \Leftrightarrow\quad
C^{*}Y^{*} = C^{*}Y 
\quad \Leftrightarrow\quad Y^{*}=Y.
\end{align*}
The reduced trace is invariant under conjugation so \(\Trd(Y)=\Trd(x)\).
So conjugation by \(C\) identifies \(S_{h}\) over \(\bbQ\) with
\[
\left\{
\begin{pmatrix}
a & b \\ \bar{b} & -a
\end{pmatrix}: a\in \bbQ,\, b\in B
\right\}.
\]
We conclude that \(\dim_{\bbQ}(S_{h})=5\).
If \(Y=\begin{psmallmatrix} a & b \\ \bar{b} & -a \end{psmallmatrix}\), then we have
\(
Y^{2}=(a^{2}+\nrd(b))I_{2}.
\)
Thus, \(\Trd(Y^{2})=4(a^{2}+\nrd(b))\), and we can write \(Y^{2}=\frac{1}{4}\Trd(Y^{2})\, I_{2}\).
Since the reduced trace is invariant under conjugation, we can also write
\[
x^{2}=C^{-1}Y^{2}C
=C^{-1}\bigl(\tfrac{1}{4}\Trd(Y^{2})I_{2}\bigr)C
=C^{-1}\bigl(\tfrac{1}{4}\Trd(x^{2})I_{2}\bigr)C
=q_{0}(x)\, I_{2}.
\]
This proves (1).
Since \(q_{0}(x)=a^{2}+\nrd(b)\), this shows that \(q_{0}\) is positive definite.

Since \(x^{2}=q_{0}(x)I_{2}\), taking reduced norms, we get \(\Nrd(x)^{2}=q_{0}(x)^{4}\), so \(\Nrd(x)=\pm q_{0}(x)^{2}\).
Since \(hx\) is hermitian and \(\Nrd(h)=1\), we have \(\Nrd(x)=\Nrd(hx)={\det}_{\rmH}(hx)^{2}\geq 0\), so \(\Nrd(x)=q_{0}(x)^{2}\).
We have \(q_{0}(x)^{2}={\det}_{\rmH}(hx)^{2}\), since we must have an identity of polynomial functions, the equation \(q_{0}(x)=\pm {\det}_{\rmH}(hx)\) must hold identically with a fixed sign for all \(x\).
Put \(x_{0}=h^{-1}\rmop{diag}(s,-t)\).
Then, \(hx_{0}\) is Hermitian and \(\Trd(x_{0})=0\), and \({\det}_{\rmH}(hx_{0})=-st<0\) and we conclude the sign.
\end{proof}

Since \(\frakP\subset B\) is a lattice, \(\rmop{M}_{2}(\frakP)\) is a full \(\bbZ\)-lattice in \(\rmop{M}_{2}(B)\).
Define
\[
\scrL_{h}=\rmop{M}_{2}(\frakP)\cap S_{h}.
\]
Since \(S_{h}\) is a \(\bbQ\)-subspace, \(\scrL_{h}\) is a full \(\bbZ\)-lattice in \(S_{h}\).

If \(x\in \scrL_{h}\), every entry \(x_{ij}\in \frakP\), hence every entry of \(x^{2}\) lies in \(\frakP^{2}=p\calO\).
Thus, \(x^{2}=q_{0}(x)I_{2}\in \rmop{M}_{2}(\frakP^{2})=\rmop{M}_{2}(p\calO)\), so that \(q_{0}(x)\in p\calO\).
But \(q_{0}(x)\in \bbQ\) since it is defined by the reduced trace.
Thus, \(q_{0}(x)\in (p\calO)\cap \bbQ=p\bbZ\), that is, \(q_{0}(x)/p\in \bbZ\).

This allows us to define the \emph{integral quadratic form}
\[
q_{h}:\scrL_{h} \to \bbZ,\quad x\mapsto \frac{q_{0}(x)}{p},
\]
which is positive definite since \(q_{h}=\frac{1}{p}q_{0}|_{\scrL_{h}}\).

\subsection{Computation of the determinant}

If \(f:V\to \bbQ\) is a quadratic form, define the (half-)\emph{polarization} of \(f\),
\[
b_{f}(X,Y)= \frac{f(X+Y)-f(X)-f(Y)}{2},\quad X,Y\in V.
\]

\begin{proposition}
For every \(X\in \Herm_{2}(B)\), we have
\[
\Trd(h^{-1}X)=-4\, b_{f}(h,X).
\]
\end{proposition}

\begin{proof}
Let \(X=\begin{psmallmatrix} a & b \\ \bar{b} & d\end{psmallmatrix}\).
Compute
\[
h^{-1}X=\begin{pmatrix}
ta-r\bar{b} & tb-rd\\
-\bar{r} a +s \bar{b} & -\bar{r} b + sd
\end{pmatrix}.
\]
Then,
\begin{align*}
\Trd(h^{-1}X)
& = \trd(ta-r\bar{b})+\trd(-\bar{r}b+sd)\\
& = \trd(ta)+\trd(sd) - [\trd(r\bar{b})+\trd(\bar{r}b)]
 = 2t\,a+2s\,d -2 \trd(r\bar{b}),
\end{align*}
since \(\trd(\bar{u})=\trd(u)\) and \(\trd(uv)=\trd(vu)\).
Thus,
\[
\Trd(h^{-1}X)=2\bigl(ta+sd-\trd(r\bar{b})\bigr).
\]
Write
\[
h+X=\begin{pmatrix}
s+a & r+b \\ \bar{r}+\bar{b} & t+d
\end{pmatrix}.
\]
Then,
\[
{\det}_{\rmH}(h+X)=(s+a)(t+d)-\nrd(r+b)
=st+sd+at+ad-\nrd(r+b).
\]
Since \(\nrd(r+b)=\nrd(r)+\nrd(b)+\trd(r\bar{b})\), we see that indeed,
\begin{align*}
{\det}_{\rmH}(h+X)
& = (st-\nrd(r))+(ad-\nrd(b))+sd+at-\trd(r\bar{b})\\
& = {\det}_{\rmH}(h)+{\det}_{\rmH}(X)+sd+at-\trd(r\bar{b}).
\end{align*}
Thus,
\begin{align*}
2\, b_{f}(h,X) & = f(h+X)-f(h)-f(X)\\
 & = -{\det}_{\rmH}(h+X)+{\det}_{\rmH}(h)+{\det}_{\rmH}(X)\\
 & = -\bigl(sd+at-\trd(r\bar{b})\bigr).
\end{align*}
The result follows.
\end{proof}

Let \(\Lambda=\bbZ \oplus \frakP \oplus \bbZ\) and let \(Q:\Lambda \to \bbZ\) be defined by
\[
Q(a,b,c)= \frac{1}{p} \nrd(b) - pac,\quad (a,b,c)\in \Lambda.
\]
Define the homomorphism of free abelian groups
\[
\eta_{h}:\Lambda \to \bbZ,\quad
\eta_{h}(a,b,c)=ta+sc-\frac{1}{p}\trd(\bar{r}b).
\]

Now, for \((a,b,c)\in \bbZ\oplus \frakP\oplus \bbZ\), define
\[
\calH(a,b,c)=\begin{pmatrix}
pa & b \\ \bar{b} & pc
\end{pmatrix},\quad (a,b,c)\in \Lambda.
\]
Then, the above calculations show that (replacing \(a\) by \(pa\) and \(d\) by \(pc\))
\[
\Trd(h^{-1}\calH(v))=2p\, \eta_{h}(v),\quad v\in \Lambda.
\]
This map \(\calH\) is an isomorphism:
\[
\calH:\bbZ\oplus \frakP\oplus \bbZ \to \rmop{M}_{2}(\frakP)\cap \Herm_{2}(B),\quad
\calH(a,b,c)=\begin{pmatrix}
pa & b \\ \bar{b} & pc
\end{pmatrix}.
\]
Its inverse is
\[
\calH^{-1}\begin{pmatrix}
A & b \\ \bar{b} & D
\end{pmatrix}=\left(\frac{A}{p},b,\frac{D}{p}\right).
\]

\begin{proposition}
With the above notation, let \(K_{h}=\ker \eta_{h}\subset\Lambda\).
Then, multiplication by \(h\) induces an isometry of integral quadratic lattices
\[
\Omega_{h}:(\scrL_{h},q_{h})\stosim (K_{h},Q)
\]
given as follows:
for \(x\in \scrL_{h}\) and \((a,b,c)\in K_{h}\), we have
\[
\Omega_{h}(x)=\left(\frac{1}{p}(hx)_{11}, (hx)_{12}, \frac{1}{p}(hx)_{22}\right),\quad
\Omega_{h}^{-1}(a,b,c)=h^{-1} \begin{pmatrix}
pa & b \\ \bar{b} & pc
\end{pmatrix}.
\]
\end{proposition}

Note that, with these formulas, it is easy to see that if \(m_{h}:x\mapsto hx\) denotes multiplication by \(h\), then
\[
\Omega_{h}=\calH^{-1}\circ m_{h}:\scrL_{h} \to K_{h},\quad
\Omega_{h}^{-1}=m_{h^{-1}} \circ \calH :K_{h} \to \scrL_{h}.
\]
Thus,
\[
\calH(\Omega_{h}(x))=hx,\quad \Omega_{h}(x)=\calH^{-1}(hx),\quad
\Omega_{h}^{-1}(a,b,c)=h^{-1}\calH(a,b,c).
\]

\begin{proof}
Since \(h\in \rmop{M}_{2}(\calO)\) is hermitian positive definite with \({\det}_{\rmH}(h)=1\), we have \(h^{-1}\in \rmop{M}_{2}(\calO)\).
Since \(\frakP\) is a two-sided ideal, if \(x\in \rmop{M}_{2}(\frakP)\), then \(hx\) and \(h^{-1}x\) have entries in \(\calO\frakP=\frakP\).
Thus, left multiplication by \(h\) on \(\rmop{M}_{2}(\frakP)\) is an automorphism of \(\bbZ\)-lattices.

Since \(\frakP\cap\bbQ=p\bbZ\) and \(\frakP\) is fixed under quaternionic conjugation, the Hermitian matrices in \(\rmop{M}_{2}(\frakP)\cap \Herm_{2}(B)\) are precisely of the form
\[
\calH(a,b,c)=\begin{pmatrix}
pa & b \\ \bar{b} & pc
\end{pmatrix}\in \rmop{M}_{2}(\frakP),\quad (a,b,c)\in \Lambda.
\]
The form \(Q\) is integral because \(\nrd(\frakP)\subset p\bbZ\).
A direct computation then shows that
\[
f(\calH(a,b,c))=p\cdot Q(a,b,c).
\]

Now let \(x\in \scrL_{h}\).
Since \(hx\) is hermitian, there is a unique triple \((a,b,c)\in \Lambda\) such that \(hx=\calH(a,b,c)\).
The description of \(hx=\calH(a,b,c)\) gives immediately the formula for \(\Omega_{h}(x)\).
Thus, \(x=h^{-1}\calH(a,b,c)\), and the above equality shows that \(\Trd(x)=0\) if and only if \(\eta_{h}(a,b,c)=0\).
So \(\Omega_{h}(\scrL_{h})\subset K_{h}=\ker\eta_{h}\).
Conversely, if \((a,b,c)\in K_{h}\), then \(x=h^{-1}\calH(a,b,c)\in \rmop{M}_{2}(\frakP)\) as explained previously, and since \(hx=\calH(a,b,c)\) is Hermitian by design, we have \((hx)^{*}=hx\).
It follows that \(x^{\dagger_{h}}=x\) and \(\Trd(x)=0\), so \(x\in \scrL_{h}\).

It remains to show that \(\Omega_{h}\) is an isometry.
For \(v=(a,b,c)\), \(Q(v)=f(\calH(v))/p\).
So for \(x\in \scrL_{h}\), we have
\[
Q(\Omega_{h}(x))=
\frac{f(\calH(\Omega_{h}(x)))}{p}
=\frac{f(hx)}{p}=\frac{q_{0}(x)}{p}=q_{h}(x).
\]
Thus, \(\Omega_{h}\) is an isometry.
\end{proof}

\section{Determinant calculation}

We now compute the determinant of \((K_{h},Q)\).

For a quadratic form \(q:V\to \bbQ\), recall that we let \(b_{q}:V\times V \to \bbQ\) denote its half-polar form.

Consider the trace dual of \(\calO\) (as a \(\bbZ\)-lattice) (\cite{Voight2021}, 15.6.1):
\begin{equation}\label{defn:tracedua}
\calO^{\#}=\{z\in B \mid \trd(z\calO)\subset \bbZ\}=\{z\in B \mid \trd(\calO z)\subset \bbZ\}.
\end{equation}
This is the codifferent of \(\calO\).

\begin{proposition}\label{prop:dual}
\(\calO^{\#}=\frakP^{-1}=p^{-1}\frakP\).
\end{proposition}

\begin{proof}
Referring to \cite{Voight2021}, the different of \(\calO\) is an integral two-sided \(\calO\)-ideal, and is the inverse of \(\calO^{\#}\) (see \cite[\S\S 16.8.1, 16.8.2]{Voight2021}).
Then, \(\nrd(\calO^{\#-1})\) is equal to the reduced discriminant of \(\calO\) (\cite[\S 16.8.4]{Voight2021}), which is equal to \(p\) (\cite[\S 15.5.5, see p.~242]{Voight2021}).
So the different is the unique two-sided prime ideal \(\frakP\) above \(p\).
We conclude that \(\calO^{\#}=\frakP^{-1}\).
Finally, since \(\frakP^{2}=p\calO\), we get
\[
\frakP(p^{-1}\frakP)=p^{-1}\frakP^{2}=\calO=(p^{-1}\frakP)\frakP.
\]
Thus, \(\frakP^{-1}=p^{-1}\frakP\).
\end{proof}

The reduced trace form
\[
B\times B \to \bbQ,\quad (x,y)\mapsto \trd(xy)=\trd(yx)
\]
is symmetric and nondegenerate, so gives rise to a \(\bbQ\)-linear isomorphism
\[
B\stosim B^{\vee}=\rmop{Hom}_{\bbQ}(B,\bbQ),\quad
z\mapsto \bigl(x\mapsto \trd(zx)\bigr).
\]
Now, \(\calO^{\#}\) is the inverse image of \(\calO^{\vee}=\rmop{Hom}_{\bbZ}(\calO,\bbZ)\) under this map, so we get an isomorphism of free abelian groups by restriction:
\[
\calO^{\#}\stosim \calO^{\vee}=\rmop{Hom}_{\bbZ}(\calO,\bbZ),\quad
z\mapsto \bigl(x\mapsto \trd(zx)\bigr).
\]

\begin{proposition}
The bilinear map
\[
\langle -,-\rangle:\calO\times \frakP \to \bbZ,\quad
(u,b)\mapsto \langle u,b\rangle =\frac{\trd(\bar{u}b)}{p}
\]
is perfect.
For every integer \(q\) and every \(u\in \calO\), we have
\[
\langle u,\frakP\rangle \subset q\bbZ
\quad\Leftrightarrow\quad
u\in q\calO.
\]
\end{proposition}

\begin{proof}
We can apply \cite{Voight2021}, 15.6.7 with \(I=\calO\) so that \(\trd(I)=\bbZ\).
This states that the map
\[
D:\calO^{\#} \to \rmop{Hom}_{\bbZ}(\calO,\bbZ),\quad
\beta \mapsto D(\beta)=(\alpha\mapsto \trd(\alpha\beta))
\]
is an isomorphism.
Conjugation preserves \(\frakP\), so \(j:\frakP\to \calO^{\#}=\frac{1}{p}\frakP\), \(b\mapsto \bar{b}/p\) is a \(\bbZ\)-linear isomorphism.
The composition \(T=D\circ j:\frakP\to \calO^{\#} \to \calO^{\vee}\) is
\[
T:\frakP\to \calO^{\vee},\quad
[T(b)](a)=\frac{\trd(a\bar{b})}{p}=\frac{\trd(\bar{a}b)}{p}=\langle a,b\rangle.
\]
We dualize \(T\).
Since this is an isomorphism of finite rank free \(\bbZ\)-modules, its dual is an isomorphism \(T^{\vee}:(\calO^{\vee})^{\vee} \to \frakP^{\vee}\), \(T^{\vee}(f)=f\circ T\).
So for \(b\in \frakP\), we have \(T^{\vee}(f)(b)=f(T(b))\).
Now the evaluation isomorphism \(\rmop{ev}_{\calO}:\calO \to (\calO^{\vee})^{\vee}\), \(a\mapsto \rmop{ev}_{a}\), where \(\rmop{ev}_{a}(\theta)=\theta(a)\) for \(\theta\in \calO^{\vee}\) is an isomorphism of \(\bbZ\)-modules.
Then,
\(
\Phi=T^{\vee}\circ \rmop{ev}_{\calO}:\calO\to \frakP^{\vee}
\)
is an isomorphism, and for \(b\in \frakP\),
\[
\Phi(a)(b)=T^{\vee}(\rmop{ev}_{a})(b)=\rmop{ev}_{a}(T(b))=[T(b)](a)=\langle a,b\rangle.
\]
So we get the isomorphism \(\Phi:\calO \to \frakP^{\vee}\).

For clarity, if \(a\in \calO\), we denote \(\Phi(a)\) by \(\Phi_{a}\).
Then, we have
\begin{align*}
\langle a,\frakP\rangle\subset q\bbZ
&\quad\Leftrightarrow\quad \Phi_{a}(\frakP)\subset q\bbZ\\
&\quad\Leftrightarrow\quad \Phi_{a}=\Phi_{qa_{0}} ~\text{for some}~a_{0}\in \calO
\quad\Leftrightarrow\quad a\in q\calO.\qedhere
\end{align*}
\end{proof}

\begin{proposition}
The restriction \(\eta_{h}:\Lambda\to \bbZ\) is primitive, that is, \(\eta_{h}(\Lambda)=\bbZ\).
\end{proposition}

\begin{proof}
The image is nontrivial since \(\eta_{h}(1,0,0)=t>0\).
Assume that \(\eta_{h}(\Lambda)\subset q\bbZ\) for some integer \(q\).
Then, from the identity \(\Trd(h^{-1}\calH(a,b,c))=2p\, \eta_{h}(a,b,c)\), we would get
\begin{align*}
t & = \eta_{h}(1,0,0)\in q\bbZ,\\
s & = \eta_{h}(0,0,1)\in q\bbZ,\\
-\langle r,b\rangle & = \eta_{h}(0,b,0)\in q\bbZ\quad (b\in \frakP).
\end{align*}
The last identity means that \(\langle r,\frakP\rangle \subset q\bbZ\).
But this holds if and only if \(r\in q\calO\).
So we conclude the existence of \(r_{0}\in \calO\) and \(s_{0},t_{0}\in \bbZ\) such that \(s=qs_{0}\), \(r=qr_{0}\) and \(t=qt_{0}\).
Then,
\(
1=st-\nrd(r)=q^{2}\bigl(s_{0}t_{0}-\nrd(r_{0})\bigr)\in q^{2}\bbZ,
\)
which implies \(q=\pm 1\), so \(\eta_{h}(\Lambda)=\bbZ\).
\end{proof}

\subsection{The Gram determinant}

For a bilinear form \(\beta\) and for a quadratic lattice \((L,q)\), we define the respective \emph{determinants} by
\[
\Delta(L,\beta)=\det(\beta(e_{i},e_{j})),\quad
\Delta(L,q)=\det(b_{q}(e_{i},e_{j})).
\]
(Here, \(b_{q}\) denotes the half-polar form of \(q\).)

\begin{proposition}
For \(\Lambda=\bbZ\oplus \frakP\oplus \bbZ\) as before, we have
\[
\Delta(\calO,\nrd)=\Delta(\frakP,\nrd/p)=\frac{p^{2}}{16},\quad
\Delta(\Lambda,Q)=-\frac{p^{4}}{64}.
\]
\end{proposition}

\begin{proof}
For \(u,v\in B\), we have \(\nrd(u+v)-\nrd(u)-\nrd(v)=\trd(u\bar{v})\), so
\(
b_{\nrd}(u,v)=\frac{1}{2}\trd(u\bar{v}).
\)
Let \((e_{i})_{1\leq i\leq 4}\) be a \(\bbZ\)-basis of \(\calO\) and let \(T=(\trd(e_{i}e_{j}))_{i,j}\).
Since \(\calO\) is a maximal order, its reduced discriminant is \(p\), the squarefree discriminant of \(B\), and the nonreduced discriminant of \(\calO\), which is equal to \(|\det T|\), is \(p^{2}\).
So \(|\det T|=p^{2}\).
To obtain the determinant of \(b_{\nrd}\) as above, we need conjugation.
This is a linear automorphism, and an involution. Let \(J\in \rmop{GL}_{4}(\bbZ)\) denote its matrix, where \(J^{2}=I_{4}\).
Hence, \(\det J=\pm 1\), and by definition, \(\bar{e}_{j}=\sum_{k} J_{kj}e_{k}\).
Now, we see that
\[
|\det(b_{\nrd}(e_{i},e_{j}))|=\bigl|\det\bigl(\tfrac{1}{2}TJ\bigr)\bigr|
=2^{-4} \, |\det T|\, |\det J| = \frac{p^{2}}{16}.
\]
Since the reduced norm is positive definite, this determinant is \(p^{2}/16\).
The map
\[
(\calO,\nrd) \xrightarrow{u\mapsto \pi u} (\frakP,\nrd/p)
\]
is a \(\bbZ\)-linear isometry:
it is bijective, since \(\frakP=\pi\calO\) and \(\nrd(\pi u)/p=\nrd(\pi)\nrd(u)/p=\nrd(u)\).
In the bases \((e_{i})\) and \((\pi e_{i})\), the Gram matrices are equal.

For the last computation, recall that we have an isomorphism of \(\bbZ\)-modules \(\calH:\Lambda \to \calH(\Lambda)\).
We also established that \(f(\calH(v))=p\, Q(v)\) for all \(v\in \Lambda\).
Since \(\rmop{rank}_{\bbZ}(\Lambda)=6\), we immediately get
\(
\Delta(\calH(\Lambda),f)=p^{6}\, \Delta(\Lambda,Q).
\)
Now,
\[
L_{\calH}=\calH(\Lambda)=\left\{
\begin{pmatrix}
pa & b \\ \bar{b} & pc
\end{pmatrix}
: a,c\in \bbZ,\, b\in \frakP
\right\}.
\]
Since \(f=-{\det}_{\rmH}\),
\[
f\begin{pmatrix}
pa & b \\ \bar{b} & pc
\end{pmatrix}=
\nrd(b)-(pa)(pc),
\]
so we see that we have an orthogonal decomposition of quadratic lattices
\[
(L_{\calH},f) \cong (\frakP,\nrd)\perp \bigl(p\bbZ\oplus p\bbZ,(x,z)\mapsto -xz\bigr).
\]
Since we already computed \((\frakP,\nrd/p)=p^{2}/16\), the scaling implies
\[
\Delta(\frakP,\nrd)=p^{4}\frac{p^{2}}{16}=\frac{p^{6}}{16}.
\]
For the hyperbolic plane, in the standard basis of \(\bbZ\oplus\bbZ\), the half-polar Gram matrix is \(\begin{psmallmatrix} 0 & -1/2\\ -1/2 & 0\end{psmallmatrix}\), so \(\Delta(\bbZ^{2},(x,z)\mapsto -xz)=-1/4\).
Since our lattice \(p\bbZ\oplus p\bbZ\) has index \(p^{2}\), a simple change of bases shows that the Gram determinant is scaled by \(p^{4}\).
Hence,
\(
\Delta(p\bbZ\oplus p\bbZ,(x,z)\mapsto -xz)=-p^{4}/4.
\)
We conclude that
\[
\Delta(L_{\calH},f)=\frac{p^{6}}{16}\cdot \left(-\frac{p^{4}}{4}\right)=-\frac{p^{10}}{64}.
\]
Dividing by \(p^{6}\) gives the result.
\end{proof}

Finally, we bring this back to the quadratic lattices \((\scrL_{h},q_{h})\) and \((K_{h},Q)\).

\begin{proposition}
Let \(w=(s,pr,t)\in \Lambda\).
Then, \(K_{h}=\Lambda \cap w^{\perp}\), and \(K_{h}\oplus \bbZ w\) is an index-\(2\) sublattice of \(\Lambda\), and
\[
\Delta(\scrL_{h},q_{h})=\Delta(K_{h},Q)=\frac{p^{3}}{16}.
\]
\end{proposition}

\begin{proof}
Let \(w=(s,pr,t)\), where \(pr\in p\calO=\frakP^{2}\subset \frakP\).
Then, \(\calH(w)=p\, h\).
Since \(Q=p^{-1}f\circ \calH\) and \(f(h)=-1\), we have
\[
Q(w)=p^{-1}\, f(\calH(w))=p^{-1}f(ph)=pf(h)=-p.
\]
We can relate the half-polar forms of \(Q\) and \(f\), and in particular, for \(w\) as above and any \(v\in \Lambda_{\bbQ}\), we have
\[
b_{Q}(v,w)=\frac{1}{p} b_{f}(\calH(v),\calH(w))=b_{f}(\calH(v),h).
\]
Since \(-4\, b_{f}(\calH(v),h)=2p\, \eta_{h}(v)\), we conclude that
\[
b_{Q}(v,w)=-\frac{p}{2}\eta_{h}(v).
\]
It follows that
\[
v\in K_{h}
\quad\Leftrightarrow\quad
\eta_{h}(v)=0
\quad\Leftrightarrow\quad
b_{Q}(v,w)=0.
\]
Thus, \(K_{h}=\Lambda \cap w^{\perp}\).
Since \(b_{Q}(w,w)=Q(w)=-p\), the expression of \(b_{Q}(v,w)\) above implies \(-p=-(p/2)\eta_{h}(w)\), hence \(\eta_{h}(w)=2\).

Since \(\eta_{h}:\Lambda \to \bbZ\) is surjective, we have an isomorphism \(\Lambda/K_{h}  \stosim \bbZ\), \(v+K_{h} \mapsto \eta_{h}(v)\).
Since \(nw\in K_{h}\) implies \(n=0\), the sum \(K_{h}+\bbZ w=K_{h}\oplus \bbZ w\) is direct.
Hence, for any \(z\in \bbZ\), \(\eta_{h}(zw+K_{h})=z\eta_{h}(w)=2z\), so \(\eta_{h}(K_{h}\oplus\bbZ w)=2\bbZ\).
It follows that \([\Lambda:K_{h}\oplus\bbZ w]=2\).

We can now complete the proof.
Since \(K_{h}\oplus \bbZ w\) is an index-\(2\) sublattice of \(\Lambda\), we have
\(
\Delta(K_{h}\oplus \bbZ w,Q)=2^{2}\Delta(\Lambda,Q)=4\,\Delta(\Lambda,Q).
\)
But since the sum is orthogonal, the Gram matrix is block diagonal, so
\[
\Delta(K_{h}\oplus \bbZ w,Q) = \Delta(K_{h},Q) \cdot \Delta(\bbZ w,Q) 
=-p\, \Delta(K_{h},Q).
\]
We conclude that
\[
\Delta(K_{h},Q)=\frac{\Delta(K_{h}\oplus \bbZ w,Q)}{-p}
=\left(-\frac{4}{p}\right)\left(-\frac{p^{4}}{64}\right)=\frac{p^{3}}{16}.
\]
Since \(\Delta(\scrL_{h},q_{h})=\Delta(K_{h},Q)\), this completes the proof.
\end{proof}

\subsection{Reconstructing the polarized isogeny}

By the definition of Hermite's constant, we have for any positive definite rank-five quadratic lattice
\[
\min_{v\in L\setminus\{0\}} q(v) \leq \gamma_{5} \,\ \Delta(L,q)^{1/5},\quad\text{where}~ \gamma_{5}=2^{3/5}.
\]
(See \cite{Cassels1959}, page 332 for the value of \(\gamma_{5}\) and IX.7 for the definition of Hermite's constant.)

Our construction immediately yields

\begin{corollary}
There exist a vector \(x\in \scrL_{h}\) and a positive integer \(N\in \bbZ_{>0}\) such that
\[
x^{2}=pN\, I_{2},\quad
N=q_{h}(x)\leq \left(\frac{p^{3}}{2}\right)^{1/5}<p^{3/5}<p.
\]
\end{corollary}

\begin{proof}
We established that \(x^{2}=q_{0}(x)I_{2}=p\, q_{h}(x)I_{2}\).
By Hermite's theorem, there exists a nonzero vector \(x\in \scrL_{h}\) such that \(q_{h}(x)=N\) is a positive integer, with
\[
1\leq N=q_{h}(x)\leq 2^{3/5} \left(\frac{p^{3}}{16}\right)^{1/5}.\qedhere
\]
\end{proof}

We can now complete the proof of Theorem \ref{theorem:A}.

\begin{proof}[Proof of Theorem \ref{theorem:A}]
Let \(x\in \scrL_{h}\) and \(N\) be as in the corollary.
Recall that \(\scrL_{h}=\rmop{M}_{2}(\frakP)\cap S_{h}\).
Now recall the diagram \eqref{eq:big-diagram}.
We reverse the arrows and obtain
\[
\begin{tikzcd}[row sep=3em,column sep=6em]
\rmop{Hom}_{ k }(A,A^{\sigma}) \rar{\phi \mapsto F\circ \phi} & \rmop{End}_{ k }(A) \\
\rmop{M}_{2}(\calO)  \uar{\gamma\mapsto (\iota^{\sigma})^{-1}\gamma \iota}& \rmop{M}_{2}(\calO)\uar[swap]{x\mapsto\iota^{-1}x\iota}.\lar
\end{tikzcd}
\]
The matrix \(x\in \rmop{M}_{2}(\frakP)\) in the bottom right corner determines an integral matrix \(\gamma=\Pi^{-1}x\in \rmop{M}_{2}(\calO)\) in the bottom left corner, and hence a homomorphism
\[
\phi=(\iota^{\sigma})^{-1}\circ \gamma\circ \iota:A \to A^{\sigma}
\]
in the top left corner.
On the other hand, define \(\xi=\iota^{-1}x\iota:A \to A\), in the top right corner.
With \(F=F_{A^{\sigma}}:A^{\sigma} \to A\), going back to the first diagram, we now have \(\iota(F\phi)\iota^{-1}=x\) and \(\xi = F\phi\), so \(x\) is the matrix of \(\xi=F\phi\).

Finally, we claim that \(\xi^{\dagger}=\xi\) and \(\xi^{2}=[pN]_{A}\).
The first identity follows immediately from the diagram \eqref{eq:Rosati-matrix}.
Indeed, \(x\) is precisely the matrix of \(F\phi\in \rmop{End}_{ k }(A)\) in \(\iota\)-coordinates.
The diagram precisely implies that \(x^{\dagger_{h}}=x\) is equivalent to \((F\phi)^{\dagger}=F\phi\).
The multiplier equation \eqref{eq:conjugate-multiplier-refined-matrix} then implies that
\[
\phi^{\dagger}\phi =[N]_{A}.
\]

It is now straightforward that \(\phi\) is an isogeny.
For every commutative \( k \)-algebra \(R\) and \(P\in (\ker\phi)(R)\), we have \([N]_{A}(P)=\phi^{\dagger}\phi(P)=0\), so \(\ker\phi\subset A[N]=\ker[N]_{A}\) as closed subgroup schemes, and \(A[N]\) is finite because \([N]_{A}\) is an isogeny (\cite{Milne1986}, Theorem 8.2), thus \(\ker\phi\) is finite.
Since \(A\) and \(A^{\sigma}\) both have dimension \(2\), \(\phi\) is an isogeny (\cite{Milne1986}, Proposition 8.1(c)).
Taking degrees in \(\phi^{\vee}\theta^{\sigma}\phi=N\theta\), since \(\theta\) and \(\theta^{\sigma}\) are principal, they have degree \(1\), and \(\deg\phi^{\vee}=\deg\phi\) , so
\[
(\deg \phi)^{2}=\deg[N]_{A}=N^{2\dim A} = N^{4}.
\]
The latter follows from \cite{Milne1986}, Theorem 8.2.
We conclude that \(\deg\phi=N^{2}\).
The finite extension \( k (A)/\phi^{*} k (A^{\sigma})\) has inseparable degree a power of \(p\) dividing \([ k (A):\phi^{*} k (A^{\sigma})]=\deg\phi\).
Since \(N<p\), we have \(p\nmid N\), so \(p\nmid N^{2}\), hence that degree is \(1\) and we conclude that \(\phi\) is separable.
\end{proof}

\part{The refined Humbert invariant}
\label{part:2}

We now introduce the refined Humbert invariant to give another proof of Theorem \ref{theorem:A} and relate the present work to the broader literature in arithmetic geometry.
We keep the same notations and conventions as in Part \ref{part:1}, unless explicitly stated.

The refined Humbert invariant of a principally polarized superspecial abelian surface $(A,\theta)$ is a positive-definite integral quadratic form in $5$ variables and provides a natural higher-dimensional analogue of the rank-$3$ Gross lattice of a supersingular elliptic curve. The parallel analogue can be summarized as follows:
\[
\begin{array}{c|c}
\text{Supersingular elliptic curve }E
&
\text{Principally polarized superspecial surface }(A,\theta)
\\ \hline
\End(E)/\mathbf Z
&
\End_\theta(A)/\mathbf Z\operatorname{id}_A
\\
\text{rank-$3$ Gross lattice}
&
\text{rank-$5$ lattice of refined Humbert invariant}
\\
\text{primitive rank-$2$ sublattices}
&
\text{primitive rank-$4$ sublattices}
\\
E\rightarrow E^\sigma
&
(A,\theta)\rightarrow(A^\sigma,\theta^\sigma).
\end{array}
\]
For supersingular elliptic curves, Aubry--Vincent relate primitive rank-$2$ Gross sublattices to trace-zero inseparable endomorphisms \cite{AubryVincent26}, and Aubry--Oyono--Vincent use this relation to obtain a bound of order $p^{1/3}$ for an isogeny to the Frobenius conjugate \cite{AubryOyonoVincent2026}. We establish the corresponding rank-$5$ construction for principally polarized superspecial abelian surfaces using the refined Humbert invariant.

The proof follows in three main steps. First, we identify the integral trace-zero realization of the refined Humbert invariant. Second, we identify its scaled dual with the Frobenius-divisible trace-zero Rosati-symmetric endomorphisms. Third, a primitive corank one sublattice is converted into a polarized isogeny to $(A^\sigma,\theta^\sigma)$. The resulting normalized fourth compound form has determinant $p^3/16$, so the rank-$5$ Hermite constant gives
$N\le (p^3/2)^{1/5}$
for the polarization multiplier $N$; equivalently, the isogeny degree is at most $(p^3/2)^{2/5}$.

\section{Refined Humbert invariant}
The refined Humbert invariant  $q_{(A,\theta)}$ of a principally polarized abelian surface $(A,\theta)$ is introduced by Kani in 1994~\cite{Kani94}. 
These invariants are useful to translate  geometric problems into  arithmetic problems as they encode many special properties of $(A,\theta)$.
In short, these invariants are integral quadratic forms in several variables 
attached to principally polarized abelian surfaces. 
We follow the main references \cite{KaniMJ,KaniESCI,KaniESCII}. For the applications of these invariants in isogeny-based cryptography, we refer to
\cite{KirimliMartindale,KirimliKorpal25, KirimliKorpal26}.

Let $A$ be an abelian surface over a field $k$ and let $\theta$ be a principal polarization. We define the \emph{polarized N\'eron-Severi group} of $(A,\theta)$ to be  $\mathcal{R_\theta}:=\NS(A)/\Z\theta$.
Kani showed in \cite[\S 3]{Kani94} that there is a well-defined map on $\NS(A)$, and this leads to a quadratic form $\tilde{q}_{(A,\theta)}$ on $\NS(A)$, given explicitly by the formula
$\tilde{q}_{(A,\theta)}(D) = (D \cdot \theta)^2 - 2(D \cdot D)$ for  $D \in \NS(A).$
It is evident that $\tilde{q}_{(A,\theta)}(D + n\theta) = \tilde{q}_{(A,\theta)}(D)$ for all $n\in \Z$. Consequently, $\tilde{q}_{(A,\theta)}$ induces  a positive-definite quadratic form $q_{(A,\theta)}$ on the quotient module $\mathcal{R_\theta}$.\footnote{This can be viewed as an integral quadratic form on $\Z^{\rho-1}$ since $\mathcal{R_\theta} \simeq Z^{\rho-1}$, where $\rho$ is the Picard number of $A$.}

Let $(A,\theta)$ be a principally polarized abelian surface. A
    \emph{refined Humbert invariant} $q_{(A,\theta)}$ of $(A,\theta)$ is a positive-definite integral quadratic form in $\rho-1$ variables on $\mathcal{R}_\theta$, 
    where $\rho$ is the Picard number of $A$, satisfying,
    for all $D\in \mathcal{R_\theta}$,
    \[q_{(A,\theta)}(D) = (D \cdot\theta)^2 - 2(D\cdot D).\]

Kani subsequently gave a different equivalent definition of a refined Humbert invariant in \cite[Remark 16]{KaniESCI}. Note that this definition can be extended to any dimension $g$.

Let $A_{/k}$ be an abelian surface. If $A$ has a principal polarization $\theta :A \rightarrow A^{\vee}$ defined over $k$,
then we define the additive subgroup $\End_{\theta}(A)$ of the endomorphism ring $\End(A)$ by:
$$\End_{\theta}(A) = \{\mu \in \End(A) :\mu^{\vee} \circ \theta = \theta \circ \mu \} .$$ 

We can identify $\NS(A)$ with a subgroup of
$\End(A)$ \cite[p.~126]{Milne1986}.
By \cite[pp.~189--190]{Mumford2008}, the map
$D \mapsto \theta^{-1}\phi_D$ induces an isomorphism
\begin{equation}\label{morphismEndNS}
\Phi_{\theta} : \NS(A) \xrightarrow{\sim} \End_{\theta}(A) := \{ \mu \in \End(A) : \mu^{\dagger_\theta} = \mu \}.
\end{equation}
We can now give the equivalent definition as follows.
\begin{proposition}\label{RHInewdefn}
Let $(A,\theta)$ be a principally polarized abelian surface over $k$, and let $q_{(A,\theta)}$ be the refined Humbert invariant of $(A,\theta)$.
Then for every $\mu \in \End_{\theta}(A)/ \Z\mathrm{id}_A$, we define
\begin{equation}\label{eq:RHI-end}
    q_{(A,\theta)}(\mu) = \tr(\mu^2) - \frac{1}{4}(\tr(\mu))^2, \end{equation}
where  $\tr$ is the usual rational trace of an endomorphism as defined in \emph{\cite[p.182]{Mumford2008}}. 
\end{proposition}
This result first appeared in \cite [Proposition 14, Remark 16]{KaniESCI}, later its proof for  was given over $\mathbb{C}$ \cite[Corollary 9.2.3]{harunkirthesis}, and later proved for algebraically closed fields \cite[Proposition 2.3]{LorenzoRitzenthalerRodriguezVillegas24}.
Let $A=E_1 \times E_2$ be a product surface with a principal polarization $\theta$. If the endomorphism rings of $E_1$ and $E_2$ are  maximal orders in the quaternion algebra $B$ ramified at $p$ and $\infty$, then its refined Humbert invariant $q_{(A,\theta)}$ is a \emph{quintic refined Humbert invariant}\footnote{A quintic quadratic form is a quadratic form in $5$ variables. In the literature, it is mostly called a quinary quadratic form; we follow Kani's and Goren-Lauter's terminology.}.

Let $\End^0(A):=\End(A)\otimes_\Z\Q$.
Let $(L,q)$ be an integral quadratic lattice of rank $r$. We now use the bilinear form convention
\begin{equation}\label{eq:beta}
  \beta_q(x,y):=q(x+y)-q(x)-q(y).
\end{equation} Thus, for a basis $e_1,\ldots,e_r$, if $M_q=\bigl(\beta_q(e_i,e_j)\bigr)_{i,j}$, then $q(x)=\frac{1}{2}x^{\mathsf T}M_qx$. We write $\det(L,q):=\det(M_q)$; this is independent of the chosen $\Z$-basis. The corresponding Gram matrix is
\begin{equation}\label{eq:Gram}
  G_q:=\frac{1}{2}M_q
  =\left(\frac{1}{2}\beta_q(e_i,e_j)\right)_{i,j},
\end{equation}
and we write $\Delta(L,q):=\det(G_q)=2^{-r}\det(L,q)$.\footnote{
This distinction is essential; the determinant $2^9p^2$ in the refined Humbert invariant convention becomes $16p^2$ for the Gram matrix used in Hermite’s inequality later.}

Since $A$ is superspecial, $\rank\NS(A)=6$. By the determinant formula for the refined Humbert invariant in \cite[Lemma 30]{KaniESCI}, together with the fact that the intersection pairing on $\NS(A)$ has determinant $-p^2$ in the superspecial case, one obtains
\begin{equation}\label{eq:RHI-det}
  \det(\mathcal{R_\theta},q_{(A,\theta)})
  =\frac{1}{2}(-4)^5(-p^2)
  =2^9p^2.
\end{equation} Hence
\begin{equation}\label{eq:RHI-Delta}
  \Delta\bigl(\mathcal{R_\theta},q_{(A,\theta)}\bigr)
  =2^{-5}\cdot2^9p^2
  =16p^2.
\end{equation}

\section{The trace-zero realization of the refined Humbert invariant}

We now compare the trace-zero realization of Lorenzo Garc\'ia--Ritzenthaler--Rodr\'iguez Villegas \cite{LorenzoRitzenthalerRodriguezVillegas24} with the integral lattice of Goren--Lauter \cite{GorenLauter06}, keeping the normalization explicit.
We define the trace removal map
\[
  \tau:\End_\theta(A)\rightarrow\End_\theta(A),\quad
  \tau(\mu):=2\mu-\frac{\tr(\mu)}2\,\id_A,
\]
and we set
\[
  \Lambda_\theta:=\tau(\End_\theta(A)),\quad
  q^0_{(A,\theta)}(x):=\frac14\tr(x^2)\text{ for } x\in\Lambda_\theta.
\]
Every $x\in\Lambda_\theta$ has rational trace zero, so \eqref{eq:RHI-end} gives $q_{(A,\theta)}([x])=\tr(x^2)=4q^0_{(A,\theta)}(x)$.
For every $\mu\in\End_\theta(A)$, the trace $\tr(\mu)$ is even, so $\tau(\mu)$ is integral; see \cite[Remark~2.2]{LorenzoRitzenthalerRodriguezVillegas24}. Moreover, we obtain $\tr(\tau(\mu))=0$.

Lorenzo Garc\'ia--Ritzenthaler--Rodr\'iguez Villegas \cite{LorenzoRitzenthalerRodriguezVillegas24} prove that $\ker\tau=\Z\id_A$ and, in dimension $g$, that the map induced by $\tau$ identifies Kani's quotient with the corresponding trace-zero lattice after scaling the quadratic form by $1/g^2$. For $g=2$, the induced map
\begin{equation}\label{eq:LGRV-isom}
  \overline\tau:\End_\theta(A)/\Z\id_A\xrightarrow{\sim}\Lambda_\theta,
  \quad [\mu]\mapsto\tau(\mu),
\end{equation}
is an isometry from $q_{(A,\theta)}$ to $q^0_{(A,\theta)}$, equivalently,
\begin{equation}\label{eq:tau-isometry}
  q^0_{(A,\theta)}\bigl(\tau(\mu)\bigr)=q_{(A,\theta)}([\mu])
  \text{ for } \mu\in\End_\theta(A).
\end{equation}
Together with $\mathcal{R_\theta}\simeq\End_\theta(A)/\Z\id_A$, this gives $(\Lambda_\theta,q^0_{(A,\theta)})\simeq(\mathcal{R_\theta},q_{(A,\theta)})$.
See \cite[Section~2.2]{LorenzoRitzenthalerRodriguezVillegas24}.

Let $\theta_0$ denote the product principal polarization on $A$. We use the corresponding notation $q_{(A,\theta_0)}$, $\Lambda_{\theta_0}$, $q^0_{(A,\theta_0)}$, and $q^{\mathrm{GL}}_{(A,\theta_0)}$.

In these coordinates, Goren--Lauter \cite[(4.10)--(4.11)]{GorenLauter06} define
\[
  \Lambda_\theta^{\mathrm{GL}}
  :=\{x\in\End_\theta(A):\tr(x)=0\}\cap\bigl(\Z\id_A+2\End_\theta(A)\bigr).
\]
 They equip this lattice with the positive-definite quadratic form $q^{\mathrm{GL}}_{(A,\theta)}(x):=\sqrt{\deg(x)}$ for $x\in\Lambda_\theta^{\mathrm{GL}}$, where $\deg(x)$ is the degree of the endomorphism $x$; see \cite[(4.12)]{GorenLauter06}.

\begin{proposition}\label{prop:trace-zero-identification}
The two integral trace-zero lattices coincide 
$\Lambda_\theta=\Lambda_\theta^{\mathrm{GL}}$.

Moreover, we have
\begin{equation}\label{eq:trace-zero-isometry}
  (\Lambda_\theta,q^0_{(A,\theta)})
  \simeq
  (\mathcal{R_\theta},q_{(A,\theta)}),
\end{equation}
and the quadratic forms agree on the common trace-zero lattice:
$  q^{\mathrm{GL}}_{(A,\theta)}(x)=q^0_{(A,\theta)}(x)$ for $x\in\Lambda_\theta$.
\end{proposition}

\begin{proof}
Let $\mu\in\End_\theta(A)$. Since $\tr(\id_A)=4$, we obtain $\tr(2\mu-\frac{\tr(\mu)}2\,\id_A)=0$, and $\tau(\mu)=2\mu-\frac{\tr(\mu)}2\,\id_A\in\Z\id_A+2\End_\theta(A)$. Thus $\Lambda_\theta\subseteq\Lambda_\theta^{\mathrm{GL}}$.

Conversely, let $x\in\Lambda_\theta^{\mathrm{GL}}$, then we have $x=n\id_A+2\mu$ for some $n\in\Z$ and $\mu\in\End_\theta(A)$, and $\tr(x)=0$. Hence $0=4n+2\tr(\mu)$ and $n=-\frac{\tr(\mu)}2$, so $x=\tau(\mu)$. Therefore, we obtain $\Lambda_\theta=\Lambda_\theta^{\mathrm{GL}}$.

The isometry \eqref{eq:trace-zero-isometry} follows from \eqref{eq:LGRV-isom} and \eqref{eq:tau-isometry}, together with $\mathcal{R_\theta}\simeq\End_\theta(A)/\Z\id_A$ through the map \eqref{morphismEndNS}.

It remains to compare the quadratic forms. For the canonical product polarization $\theta_0$ on $A=E\times E$, we obtain from \cite[(4.5)]{GorenLauter06}:
\begin{equation}\label{eq:product-Lambda}
  \Lambda_{\theta_0}
  =\left\{
  x(a,\omega):=
  \begin{pmatrix}
    a&2\omega\\
    2\overline\omega&-a
  \end{pmatrix}
  :a\in\Z,\ \omega\in\mathcal{O}
  \right\},
\end{equation}
and \cite[(4.12)]{GorenLauter06} gives
\begin{equation}\label{eq:qGL-product}
  q^{\mathrm{GL}}_{(A,\theta_0)}(x(a,\omega))=a^2+4\nrd(\omega).
\end{equation}
For such an element, $x(a,\omega)^2=\bigl(a^2+4\nrd(\omega)\bigr)I_2$. Since the rational trace of $tI_2$ is $4t$, the definition of $q^0_{(A,\theta_0)}$ gives
\[
  q^0_{(A,\theta_0)}(x(a,\omega))
  =\frac14\tr\bigl(x(a,\omega)^2\bigr)
  =a^2+4\nrd(\omega).
\]
Comparing with \eqref{eq:qGL-product}, we obtain
\[
  q^{\mathrm{GL}}_{(A,\theta_0)}(x)
  =q^0_{(A,\theta_0)}(x)
  \quad \text{ for }x\in\Lambda_{\theta_0}.
\]
Let $V_{\theta_0}:=\Lambda_{\theta_0}\otimes_\Z\Q$. Both quadratic forms extend uniquely to the rational space $V_{\theta_0}$. For a general principal polarization, the rational conjugations of \cite[(4.7)--(4.12)]{GorenLauter06} identify the corresponding trace-zero rational spaces and preserve both degree and rational trace. Consequently they transport both $q^{\mathrm{GL}}_{(A,\theta_0)}$ and $q^0_{(A,\theta_0)}$ to the corresponding quadratic forms on $V_\theta:=\Lambda_\theta\otimes_{\mathbb Z}\mathbb{Q}$. The equality therefore holds on $V_\theta$, and hence on the common integral lattice $\Lambda_\theta$.
\end{proof}

A closely related rational construction also appears in Kudla--Rapoport \cite{KR00}. Their special endomorphisms are trace-zero Rosati-symmetric endomorphisms and satisfy
$\mu^2=q(\mu)\operatorname{id}_A$,
see \cite[Definition~2.1 and Lemma~2.2]{KR00}. In \cite[Appendix~A.3]{KR00}, the corresponding five-dimensional quadratic space is described explicitly inside $\rmop{M}_{2}(B)$, and its even Clifford algebra is identified with $\rmop{M}_{2}(B)$. Thus, after passing to $\Q$, this is the same quaternionic quadratic space as $V_\theta$ above. At the integral level, however, Kudla--Rapoport work with a slightly larger trace-zero lattice containing the lattice $\Lambda_\theta$ considered here; see also~\cite[p. 4]{LorenzoRitzenthalerRodriguezVillegas24}.
\begin{remark}
Lorenzo Garc\'ia--Ritzenthaler--Rodr\'iguez Villegas \cite{LorenzoRitzenthalerRodriguezVillegas24} use the half-polar bilinear form $b_q(x,y)=\frac{1}{2}\bigl(q(x+y)-q(x)-q(y)\bigr)$ as in Part \ref{part:1} of the present article. So in Part \ref{part:2}, we have $b_q=\beta_q/2$ in their convention. The isometry in \eqref{eq:LGRV-isom} is independent of this, but the factor is important when the dual lattice is involved.
\end{remark}

From Proposition~\ref{prop:trace-zero-identification} and \eqref{eq:RHI-det}--\eqref{eq:RHI-Delta}, we have
\begin{equation}\label{eq:L-det}
  \det(\Lambda_\theta,q^0_{(A,\theta)})=2^9p^2,
  \quad
  \Delta(\Lambda_\theta,q^0_{(A,\theta)})=16p^2.
\end{equation}

\section{The dual lattice and the Frobenius ideal}\label{sec:dual}
Let $\End^0_\theta(A):=\End_\theta(A)\otimes_{\mathbb{Z}}\mathbb{Q}$.
Recall $V_\theta:=\Lambda_\theta\otimes_\Z\Q=\{x\in\End_\theta^0(A):\tr(x)=0\}$. We use the same symbol $q^0_{(A,\theta)}$ for the $\Q$-valued extension of the quadratic form to $V_\theta$, and likewise extend its bilinear form $\beta_{q^0_{(A,\theta)}}$. We use the same convention for $q^0_{(A,\theta_0)}$ on $V_{\theta_0}$. In accordance with \eqref{eq:Gram}, the dual lattice introduced above is
\[
  \Lambda_\theta^{\#}
  :=\left\{y\in V_\theta:
  \frac{1}{2}\beta_{q^0_{(A,\theta)}}(y,\Lambda_\theta)\subseteq\Z\right\}.
\]
Equivalently, $\beta_{q^0_{(A,\theta)}}(y,\Lambda_\theta)\subseteq2\Z$. Thus $\Lambda_\theta^{\#}$ is the dual of $\Lambda_\theta$ for the Gram bilinear form $\beta_{q^0_{(A,\theta)}}/2$.
We first compute the dual for the product polarization $\theta_0$ on $A$.
An element of $V_{\theta_0}$ has the form
\[
  y(s,u):=
  \begin{pmatrix}
    s&u\\
    \overline u&-s
  \end{pmatrix},\quad s\in\Q,\ u\in B,
\]
and $q^0_{(A,\theta_0)}(y(s,u))=s^2+\nrd(u)$. Then, we have
\begin{equation}\label{eq:product-pairing}
  \frac{1}{2}\beta_{q^0_{(A,\theta_0)}}\bigl(y(s,u),x(a,\omega)\bigr)
  =sa+\trd(u\overline\omega).
\end{equation}
With $F=F_{A^\sigma}$ as in Section~\ref{sec:Frob}, we define
\[
I_F(A):=F\circ\Hom_k(A,A^\sigma)
\subseteq\End_k(A).
\]
By the commutative diagram \eqref{eq:big-diagram} and the discussion following it, $I_F(A)$ corresponds to $\rmop{M}_{2}(\mathfrak{P})$ under the identification of Lemma \ref{lemma:IKO}.
\begin{lemma}\label{lem:product-dual}
For the canonical product polarization $\theta_0$, we have
\begin{equation}\label{eq:product-dual}
  \Lambda_{\theta_0}^{\#}
  =\left\{
  \begin{pmatrix}
    a&u\\
    \overline u&-a
  \end{pmatrix}
  :a\in\Z,\ u\in\mathcal{O}^{\#}\right\}.
\end{equation}
Consequently,
\begin{equation}\label{eq:product-scaled-dual}
  p\Lambda_{\theta_0}^{\#}
  =V_{\theta_0}\cap \rmop{M}_{2}(\mathfrak{P}).
\end{equation}
\end{lemma}

\begin{proof}
Equation~\eqref{eq:product-dual} follows directly from Proposition~\ref{defn:tracedua} and \eqref{eq:product-pairing}. Multiplying by $p$ and applying Proposition~\ref{prop:dual} gives diagonal entries in $p\Z$ and off-diagonal entries in $p\mathcal{O}^{\#}=\mathfrak{P}$. Since $\mathfrak{P}\cap\Q=p\Z$ by Section~\ref{sec:GaloisRamified}, this is exactly $V_{\theta_0}\cap \rmop{M}_{2}(\mathfrak{P})$.
\end{proof}

For a general principal polarization, Goren--Lauter \cite{GorenLauter06} proved that the localized integral trace-zero quadratic lattices are conjugate. More precisely, for every rational prime $\ell$ there exists $H_\ell\in\GL_2(\mathcal{O}_\ell)$ such that $x\mapsto H_\ell^{-1}xH_\ell$ identifies $((\Lambda_\theta)_\ell,q^0_{(A,\theta)})$ with $((\Lambda_{\theta_0})_\ell,q^0_{(A,\theta_0)})$; see the discussion following \cite[(4.12)]{GorenLauter06} and Proposition~\ref{prop:trace-zero-identification}. The same conjugation identifies their duals. Since $\mathfrak{P}_\ell$ is a two-sided ideal by Section~\ref{sec:GaloisRamified}, and, after localization, $H_\ell,H_\ell^{-1}\in \rmop{M}_{2}(\mathcal{O}_\ell)$,
\begin{equation}\label{eq:P-preserved}
  H_\ell^{-1}\rmop{M}_{2}(\mathfrak{P}_\ell)H_\ell
  =\rmop{M}_{2}(\mathfrak{P}_\ell).
\end{equation}

\begin{proposition}\label{prop:scaled-dual}
For a fixed principal polarization $\theta$, we have
\begin{equation}\label{eq:scaled-dual}
  p\Lambda_\theta^{\#}=V_\theta\cap I_F(A).
\end{equation}
Under the identification of Lemma \ref{lemma:IKO}, this is equivalently
\[
  p\Lambda_\theta^{\#}=V_\theta\cap \rmop{M}_{2}(\mathfrak{P}).
\]
\end{proposition}

\begin{proof}
By using Lemma \ref{lemma:IKO}, Equation~\eqref{eq:big-diagram} identifies $I_F(A)$ with $\rmop{M}_{2}(\mathfrak{P})$. Localizing \eqref{eq:product-scaled-dual} at $\ell$ and transporting it by the Goren--Lauter \cite{GorenLauter06} local conjugation gives, using \eqref{eq:P-preserved},
\[
  p(\Lambda_\theta^{\#})_\ell
  =(V_\theta)_\ell\cap \rmop{M}_{2}(\mathfrak{P}_\ell)
\]
for every prime $\ell$. Both sides are full $\Z$-lattices in $V_\theta$, and two full lattices are equal if and only if all of their localizations are equal. This proves the coordinate identity, and hence the equality \eqref{eq:scaled-dual}.
\end{proof}

\section{From the scaled dual to a polarized isogeny}

The scaled-dual identity identifies $p\Lambda_\theta^\#$ with the trace-zero Rosati-symmetric endomorphisms lying in the Frobenius ideal. We first record the scalar-square identity on $V_\theta$.

\begin{lemma}\label{lem:scalar-square}
For every $x\in V_\theta$, we obtain
  $x^2=q^0_{(A,\theta)}(x)\,\id_A$.
\end{lemma}

\begin{proof}
For the product polarization, write $x=\begin{psmallmatrix}s&u\\ \overline u&-s\end{psmallmatrix}$. Direct multiplication gives $x^2=(s^2+\nrd(u))I_2=q^0_{(A,\theta_0)}(x)I_2$. For a general polarization, the rational Goren--Lauter conjugations of \cite[(4.7)--(4.12)]{GorenLauter06} identify the corresponding trace-zero rational spaces, preserve multiplication, and transport $q^{\mathrm{GL}}_{(A,\theta_0)}$ to $q^{\mathrm{GL}}_{(A,\theta)}$. By Proposition~\ref{prop:trace-zero-identification}, these forms are $q^0_{(A,\theta_0)}$ and $q^0_{(A,\theta)}$, respectively. Thus the identity transports to $V_\theta$.
\end{proof}

\begin{proposition}\label{prop:dual-to-isogeny}
Let $0\ne y\in\Lambda_\theta^{\#}$ and put $x=py$. Then there exists a polarized isogeny
$\varphi_y:(A,\theta)\rightarrow(A^\sigma,\theta^\sigma)$
whose multiplier is
\[
  N(y)=p\,q^0_{(A,\theta)}(y)\in\Z_{>0}.
\]
In particular,
$\varphi_y^*\theta^\sigma=N(y)\theta,
  \quad
  \deg(\varphi_y)=N(y)^2$.
\end{proposition}

\begin{proof}
By Proposition~\ref{prop:scaled-dual}, $x=py\in V_\theta\cap I_F(A)$. By the definition of $I_F(A)$, there exists $\varphi_y:A\to A^\sigma$ such that
\begin{equation}\label{eq:x=Fphi}
  x=F_{A^\sigma}\circ\varphi_y.
\end{equation}
By Lemma~\ref{lem:scalar-square},
we have $ x^2=q^0_{(A,\theta)}(x)\id_A=p^2q^0_{(A,\theta)}(y)\id_A$.

Under the quaternion coordinates of Section~\ref{sec:Coordinates} and Lemma~\ref{lemma:IKO}, the inclusion $x\in I_F(A)$ becomes $x\in \rmop{M}_{2}(\mathfrak{P})$. Since $\mathfrak{P}^2=p\mathcal{O}$ by Section~\ref{sec:GaloisRamified}
$x^2\in \rmop{M}_{2}(\mathfrak{P}^2)=p\rmop{M}_{2}(\mathcal{O})$.

The scalar coefficient $q^0_{(A,\theta)}(x)$ is rational. Since $x^2=q^0_{(A,\theta)}(x)I_2\in p\rmop{M}_{2}(\mathcal{O})$ in the auxiliary coordinates and $\mathcal{O}\cap\Q=\Z$, it follows that $q^0_{(A,\theta)}(x)\in p\Z$. Thus
\[
  N(y):=\frac{q^0_{(A,\theta)}(x)}p=pq^0_{(A,\theta)}(y)\in\Z_{>0}.
\]

Since $x\in V_\theta$, it is Rosati-symmetric, so $x^\dagger=x$. Using \eqref{eq:x=Fphi}, we have
\[
\begin{aligned}
  [pN(y)]_A
   &=x^2=x^\dagger x\\
   &=\varphi_y^\dagger\circ F_{A^\sigma}^\dagger\circ F_{A^\sigma}\circ\varphi_y\\
   &=\varphi_y^\dagger\circ[p]_{A^\sigma}\circ\varphi_y
    =[p]_A\circ\varphi_y^\dagger\circ\varphi_y.
\end{aligned}
\]
The additive group $\End(A)$ is torsion-free, hence
\[
  \varphi_y^\dagger\circ\varphi_y=[N(y)]_A.
\]
Equivalently, $\varphi_y^*\theta^\sigma=N(y)\theta$. Since $N(y)>0$, the homomorphism $\varphi_y$ is an isogeny. Taking degrees on an abelian surface gives $\deg(\varphi_y)=N(y)^2$.
\end{proof}

\section{Primitive sublattices of rank $4$}

We now return to the codimension one formulation derived from the refined Humbert invariant.

\begin{lemma}\label{lem:corank-one}
Let $(L,q)$ be a positive-definite integral quadratic lattice of rank $5$, let $\beta_q$ be as in \eqref{eq:beta}, and extend $q$ and its bilinear form to $L\otimes_\Z\Q$. Define
\[
  L^\#=\{y\in L\otimes_\Z\Q:\tfrac{1}{2}\beta_q(y,L)\subseteq\Z\}.
\]
If $M\subset L$ is a primitive rank-$4$ sublattice, then there is a primitive $y\in L^\#$, unique up to sign, such that $M=\{v\in L:\tfrac{1}{2}\beta_q(v,y)=0\}$. Moreover,
\begin{equation}\label{eq:corank-det}
  \Delta(M,q)=\Delta(L,q)q(y).
\end{equation}
\end{lemma}

\begin{proof}
Since $M$ is primitive of corank one, $L/M\simeq\Z$. Choose a surjective homomorphism $\psi:L\to\Z$ with kernel $M$. Nondegeneracy of $\beta_q/2$ gives a unique $y\in L^\#$ satisfying $\psi(v)=\frac{1}{2}\beta_q(v,y)$. Surjectivity of $\psi$ is equivalent to primitivity of $y$ in $L^\#$, and changing $\psi$ to $-\psi$ changes $y$ to $-y$.

Choose a basis $v_1,\ldots,v_4$ of $M$ and extend it to a basis $v_1,\ldots,v_4,v_5$ of $L$ with $\psi(v_5)=1$. Relative to the Gram bilinear form $\beta_q/2$, $y$ is the fifth vector in the dual basis. The Schur complement formula for the Gram matrix therefore gives
\[
  q(y)=\frac{\Delta(M,q)}{\Delta(L,q)},
\]
which is \eqref{eq:corank-det}.
\end{proof}

\begin{theorem}[Rank-$4$--isogeny correspondence]\label{thm:rank4}
Let $(\Lambda_\theta,q^0_{(A,\theta)})$ be the integral trace-zero realization of the refined Humbert invariant of the fixed surface $(A,\theta)$. For every primitive rank-$4$ sublattice $M\subset\Lambda_\theta$, there is a primitive normal vector $y_M\in\Lambda_\theta^\#$, unique up to sign. The two choices $\pm y_M$ yield polarized isogenies $\pm\varphi_M:(A,\theta)\to(A^\sigma,\theta^\sigma)$. Their common multiplier depends only on $M$ and is
\begin{equation}\label{eq:N-M}
  N(M)=\frac{\Delta(M,q^0_{(A,\theta)})}{16p}.
\end{equation}
Equivalently, in the determinant convention based on $\beta_{q^0_{(A,\theta)}}$,
\begin{equation}\label{eq:beta-det-rank4}
  \det(M,q^0_{(A,\theta)})=2^8pN(M).
\end{equation}
\end{theorem}

\begin{proof}
By Lemma~\ref{lem:corank-one}, there is a primitive $y\in\Lambda_\theta^\#$ such that $M=\{v\in\Lambda_\theta:\frac{1}{2}\beta_{q^0_{(A,\theta)}}(v,y)=0\}$ and
\[
  \Delta(M,q^0_{(A,\theta)})
  =\Delta(\Lambda_\theta,q^0_{(A,\theta)})q^0_{(A,\theta)}(y)
  =16p^2q^0_{(A,\theta)}(y),
\]
where we used \eqref{eq:L-det}. By Proposition~\ref{prop:dual-to-isogeny}, the corresponding polarized isogeny has multiplier
$N(M)=pq^0_{(A,\theta)}(y)$.
Hence, $\Delta(M,q^0_{(A,\theta)})=16pN(M)$,
which proves \eqref{eq:N-M}. Since $M$ has rank $4$, replacing the Gram matrix of $\beta_{q^0_{(A,\theta)}}/2$ by the coefficient matrix of $\beta_{q^0_{(A,\theta)}}$ multiplies its determinant by $2^4$. Thus
\[
  \det(M,q^0_{(A,\theta)})
  =2^4\Delta(M,q^0_{(A,\theta)})
  =2^8pN(M),
\]
 hence proving \eqref{eq:beta-det-rank4}.
\end{proof}

\section{The $p^{3/5}$ bound}

We now express the codimension one construction directly as a quintic quadratic form. Fix a basis $e_1,\ldots,e_5$ of $\Lambda_\theta$ and let $G=\left(\frac{1}{2}\beta_{q^0_{(A,\theta)}}(e_i,e_j)\right)_{i,j}$. By \eqref{eq:L-det}, $\det G=16p^2$.
Let $M\subset \Lambda_\theta$ be a primitive rank-$4$ sublattice, and let $T\in\Z^{5\times4}$ have as columns the coordinate vectors, in the basis $e_1,\ldots,e_5$, of a $\Z$-basis of $M$. Let $m:=(m_1,\ldots,m_5)^{\mathsf T}$ be the vector of signed $4\times4$ minors of $T$. Since $M$ is primitive, $\gcd(m_1,\ldots,m_5)=1$. Writing $\adj(G)$ for the adjugate matrix of $G$, the Cauchy--Binet formula gives
\[
  \Delta(M,q^0_{(A,\theta)})=\det(T^{\mathsf T}GT)=m^{\mathsf T}\adj(G)m.
\]
Define the Frobenius normalized fourth compound form $Q_F:\Z^5\to\Q$ by
\[
  Q_F(m):=\frac1{16p}\,m^{\mathsf T}\adj(G)m.
\]
For the primitive minor vector attached to $M$, Equation~\eqref{eq:N-M} gives $Q_F(m)=N(M)$.
Since $\adj(G)=\det(G)G^{-1}=16p^2G^{-1}$, we have $Q_F(m)=p\,m^{\mathsf T}G^{-1}m$.
Using the basis $e_1,\ldots,e_5$, the dual lattice $\Lambda_\theta^\#$ has coordinate lattice $G^{-1}\Z^5$. Hence the integer vector $m$ defines the dual vector
\[
  y_m:=\sum_{i=1}^5 (G^{-1}m)_i e_i\in\Lambda_\theta^\#,
\]
for which $q^0_{(A,\theta)}(y_m)=m^{\mathsf T}G^{-1}m$. Thus the fourth compound form is precisely the quadratic form of the scaled dual lattice $(\Lambda_\theta^\#,p q^0_{(A,\theta)})$. Its Gram determinant is
\[
  \det(pG^{-1})
  =\frac{p^5}{16p^2}
  =\frac{p^3}{16}.
\]
Equivalently, $\det(\adj G)=(\det G)^4=(16p^2)^4$, and dividing the fourth compound form by $16p$ in rank $5$ gives $\frac{(16p^2)^4}{(16p)^5}=\frac{p^3}{16}$. 

The exact Hermite constant in dimension $5$ is $\gamma_5=2^{3/5}$; see \cite[Table 3.1]{conway2013sphere}. Hermite's inequality applied to $(\Lambda_\theta^\#,p q^0_{(A,\theta)})$ therefore gives a nonzero $y\in \Lambda_\theta^\#$ such that
\[
  p q^0_{(A,\theta)}(y)
  \le
  \gamma_5(p^3/16)^{1/5}
  =(p^3/2)^{1/5}.
\]
A shortest nonzero vector of a lattice is primitive, since otherwise division by its content produces a shorter nonzero vector. Applying Proposition~\ref{prop:dual-to-isogeny} gives the explicit bound.

\begin{theorem}\label{thm:sharp}
There exists a separable polarized isogeny $\varphi:(A,\theta)\rightarrow(A^\sigma,\theta^\sigma)$ with $N$ satisfying
$  N\le(p^3/2)^{1/5}$.
Therefore, $\deg(\varphi)
  \le(p^3/2)^{2/5}$.

\end{theorem}

\begin{proof}
Choose a shortest nonzero $y\in \Lambda_\theta^\#$ for the quadratic form $p q^0_{(A,\theta)}$. The preceding Hermite estimate and Proposition~\ref{prop:dual-to-isogeny} give
\[
  N=p q^0_{(A,\theta)}(y)
  \le(p^3/2)^{1/5}.
\]
For $p>1$ this bound is strictly less than $p$, so $p\nmid N$. Since $\deg(\varphi)=N^2$, the degree is prime to $p$ and $\varphi$ is separable.
\end{proof}

\printbibliography

\end{document}